\documentclass[10pt,reqno]{amsart}

\usepackage{amssymb}
\usepackage{bm}
\usepackage{fixmath}
\usepackage{mathrsfs}

\newtheorem{theorem}{Theorem}

\newtheorem{proposition}[theorem]{Proposition}

\newcommand{\n}[1]{\left\vert#1\right\vert}

\begin{document}
\title[]{On differential independence of $\mathbold{\zeta}$ and $\mathbold{\Gamma}$}

\author[]{Qi Han and Jingbo Liu}

\address{Department of Science and Mathematics, Texas A\&M University, San Antonio, Texas 78224, USA
\vskip 2pt Email: {\sf qhan@tamusa.edu (QH) \hspace{0.6mm} jliu@tamusa.edu (JL)}}

\begin{abstract}
  In this note, we will prove that $\mathbold{\zeta}$ and $\mathbold{\Gamma}$ can not satisfy any differential equation generated through a family of functions continuous in $\mathbold{\zeta}$ with polynomials in $\mathbold{\Gamma}$.
\end{abstract}

\maketitle

It is a profound result of H\"{o}lder \cite{Ho} in 1887 that the Euler gamma-function
\begin{equation}
\mathbold{\Gamma}(z)=\int_0^{+\infty}t^{z-1}e^{-t}dt\nonumber
\end{equation}
can not satisfy any nontrivial algebraic differential equation whose coefficients are polynomials in $\mathbf{C}$.
That is, if $P(v_0,v_1,\ldots,v_n)$ is a polynomial of $n+1$ variables with polynomial coefficients in $z\in\mathbf{C}$ such that $P(\mathbold{\Gamma},\mathbold{\Gamma}',\ldots,\mathbold{\Gamma}^{(n)})(z)\equiv0$ for $z\in\mathbf{C}$, then necessarily $P\equiv0$.
Hilbert \cite{Hi1,Hi2}, in the lecture addressed before the International Congress of Mathematicians at Paris in 1900 for his famous 23 problems, stated in Problem 18 that the Riemann zeta-function
\begin{equation}
\mathbold{\zeta}(z)
\footnote{In this paper, we will use $\mathbold{\zeta}(z)$, instead of the common practice $\mathbold{\zeta}(s)$, to keep notations consistent.}
=\sum_{n=1}^{+\infty}\frac{1}{n^z}\nonumber
\end{equation}
can not satisfy any nontrivial algebraic differential equation whose coefficients are polynomials in $\mathbf{C}$.
This problem was solved by Mordukhai-Boltovskoi \cite{MB} and Ostrowski \cite{Os}, independently.
In addition, Voronin \cite{Vo2} proved, as a special case of his result, $\mathbold{\zeta}$ can not satisfy any nontrivial differential equation generated from a continuous function $F(u_0,u_1,\ldots,u_m)$ of $m+1$ variables with polynomial coefficients in $\mathbf{C}$, based on his celebrated result \cite{Vo1} as follows.

\begin{proposition}\label{Prop1}
Given $x\in\big(\frac{1}{2},1\big)$ for $z=x+\mathtt{i}y\in\mathbf{C}$, define
\begin{equation}
\gamma(y):=(\mathbold{\zeta}(x+\mathtt{i}y),\mathbold{\zeta}'(x+\mathtt{i}y),\ldots,\mathbold{\zeta}^{(m)}(x+\mathtt{i}y))\nonumber
\end{equation}
to be a curve in $y$.
Then, $\gamma(\mathbf{R})$ is everywhere dense in $\mathbf{C}^{m+1}$.
\end{proposition}

On the other hand, $\mathbold{\zeta}$ and $\mathbold{\Gamma}$ are related by the Riemann functional equation
\begin{equation}
\mathbold{\zeta}(1-z)=2^{1-z}\pi^{-z}\cos\Big(\frac{1}{2}\pi z\Big)\mathbold{\Gamma}(z)\mathbold{\zeta}(z).\nonumber
\end{equation}
In 2007, Markus \cite{Ma} showed that $\mathbold{\zeta}(\sin(2\pi z))$ can not satisfy any nontrivial algebraic differential equation whose coefficients are polynomials in $\mathbold{\Gamma}$ and its derivatives, and he conjectured $\mathbold{\zeta}$ itself can not satisfy any nontrivial algebraic differential equation whose coefficients are polynomials in $\mathbold{\Gamma}$ and its derivatives, either.
As such, one is now interested in knowing whether there exists a nontrivial polynomial $\mathcal{P}(u_0,u_1,\ldots,u_m;v_0,v_1,\ldots,v_n)$ such that, for $z\in\mathbf{C}$,
\begin{equation}\label{Eq01}
\mathcal{P}(\mathbold{\zeta},\mathbold{\zeta}',\ldots,\mathbold{\zeta}^{(m)};\mathbold{\Gamma},\mathbold{\Gamma}',\ldots,\mathbold{\Gamma}^{(n)})(z)\equiv0.
\end{equation}
Important partial results of \eqref{Eq01} can be found in Li and Ye \cite{LY1,LY2} for $n=1,2$.

In this note, we manipulate the key ideas involved in \cite{Ho,Vo2,LY2} and consider some differential equations generated through a family $\mathscr{F}$ of functions $\mathcal{F}(u_0,u_1,\ldots,u_m;v_0,v_1,\ldots,v_n)$ which are continuous in $\vec{u}:=(u_0,u_1,\ldots,u_m)$ with polynomial coefficients in $\vec{v}:=(v_0,v_1,\ldots,v_n)$.

As a matter of fact, we define
\begin{equation}
\Lambda:=\left\{\lambda:=(\lambda_0,\lambda_1,\ldots,\lambda_n):\lambda_0,\lambda_1,\ldots,\lambda_n~\text{are nonnegative integers}\right\}\nonumber
\end{equation}
to be an index set having a finite cardinality, and accordingly write
\begin{equation}
\begin{aligned}
&\Lambda_p:=\left\{\lambda\in\Lambda:\n{\lambda}:=\lambda_0+\lambda_1+\cdots+\lambda_n=p\right\}~\text{and}\\
&\Lambda^\star_q:=\left\{\lambda\in\Lambda:\n{\lambda}^\star:=\lambda_1+n\lambda_n=q\right\}.\nonumber
\end{aligned}
\end{equation}
Let $\mathcal{F}(u_0,u_1,\ldots,u_m;v_0,v_1,\ldots,v_n)$ be a function of $m+n+2$ variables such that
\begin{equation}
\mathcal{F}(u_0,u_1,\ldots,u_m;v_0,v_1,\ldots,v_n)=\sum_{\lambda\in\Lambda}F_\lambda(u_0,u_1,\ldots,u_m)v_0^{\lambda_0}v_1^{\lambda_1}\cdots v_n^{\lambda_n}.\nonumber
\end{equation}
Here, the $F_\lambda(u_0,u_1,\ldots,u_m)$ are continuous functions in $\mathbf{C}^{m+1}$.
Moreover, for fixed $\n{\lambda}$, the set $\left\{\lambda_2,\ldots,\lambda_{n-1}\right\}$ depends only on $\n{\lambda}^\star$ (that is, $\lambda_2,\ldots,\lambda_{n-1}$ are fixed provided $\n{\lambda}^\star$ is; note that, though, one does have extensive latitude in $\lambda_1$ and $\lambda_n$ to a certain degree) and
\begin{equation}\label{Eq02}
\chi(\n{\lambda}^\star):=2\lambda_2+3\lambda_3+\cdots+(n-1)\lambda_{n-1},
\end{equation}
as a function of $\n{\lambda}^\star$, is required be nondecreasing in $\n{\lambda}^\star$.
Denote by $\mathscr{F}$ the family of functions $\mathcal{F}(u_0,u_1,\ldots,u_m;v_0,v_1,\ldots,v_n)$ that satisfy the preceding properties.

Then, the main result of this paper can be formulated as follows.

\begin{theorem}\label{Thm1}
Assume $\mathcal{F}(u_0,u_1,\ldots,u_m;v_0,v_1,\ldots,v_n)$ is a function of $m+n+2$ variables that belongs to the family $\mathscr{F}$.
If, for $z\in\mathbf{C}$,
\begin{equation}\label{Eq03}
\mathcal{F}(\mathbold{\zeta},\mathbold{\zeta}',\ldots,\mathbold{\zeta}^{(m)};\mathbold{\Gamma},\mathbold{\Gamma}',\ldots,\mathbold{\Gamma}^{(n)})(z)\equiv0,
\end{equation}
then the function $\mathcal{F}$ vanishes identically.
\end{theorem}

\begin{proof}
Since $\Lambda$ has a finite cardinality, one sees, for a nonnegative integer $L_\Lambda$,
\begin{equation}
\mathcal{F}(u_0,u_1,\ldots,u_m;v_0,v_1,\ldots,v_n)=\sum_{p=0}^{L_\Lambda}\sum_{\lambda\in\Lambda_p}F_\lambda(\vec{u})v_0^{\lambda_0}v_1^{\lambda_1}\cdots v_n^{\lambda_n}.\nonumber
\end{equation}
Here, and hereafter, $F_\lambda(\vec{u})$ represents the function $F_\lambda(u_0,u_1,\ldots,u_m)$ to save space.
Write, for each $p=0,1,\ldots,L_\Lambda$, the associated homogeneous polynomial in $\vec{v}$ of degree $p$ to be
\begin{equation}
\mathcal{F}_p(u_0,u_1,\ldots,u_m;v_0,v_1,\ldots,v_n):=\sum_{\lambda\in\Lambda_p}F_\lambda(\vec{u})v_0^{\lambda_0}v_1^{\lambda_1}\cdots v_n^{\lambda_n}.\nonumber
\end{equation}
Rearranging $\mathcal{F}_p(u_0,u_1,\ldots,u_m;v_0,v_1,\ldots,v_n)$ in $\vec{v}$ in the ascending order of $q=\n{\lambda}^\star$, together with the lexicographical order, yields a nonnegative integer $M_p$ with
\begin{equation}\label{Eq04}
\mathcal{F}_p(u_0,u_1,\ldots,u_m;v_0,v_1,\ldots,v_n)
=\sum_{q=0}^{M_p}\sum_{\lambda\in\Lambda_p\cap\Lambda^\star_q}F_\lambda(\vec{u})v_0^{\lambda_0}v_1^{\lambda_1}\cdots v_n^{\lambda_n},
\end{equation}
so that
\begin{equation}\label{Eq05}
\mathcal{F}(u_0,u_1,\ldots,u_m;v_0,v_1,\ldots,v_n)
=\sum_{p=0}^{L_\Lambda}\sum_{q=0}^{M_p}\sum_{\lambda\in\Lambda_p\cap\Lambda^\star_q}F_\lambda(\vec{u})v_0^{\lambda_0}v_1^{\lambda_1}\cdots v_n^{\lambda_n}.
\end{equation}

In view of \eqref{Eq03}, for each $p=0,1,\ldots,L_\Lambda$ and all $z\in\mathbf{C}$, we claim that
\begin{equation}\label{Eq06}
\mathcal{F}_p(\mathbold{\zeta},\mathbold{\zeta}',\ldots,\mathbold{\zeta}^{(m)};\mathbold{\Gamma},\mathbold{\Gamma}',\ldots,\mathbold{\Gamma}^{(n)})(z)\equiv0.
\end{equation}

Suppose, on the contrary, $p_0$ is the least index among $\left\{0,1,\ldots,L_\Lambda\right\}$ that violates \eqref{Eq06}.
Then, we have
\begin{equation}\label{Eq07}
\begin{aligned}
&\mathcal{F}_{p_0}\Big(\mathbold{\zeta},\mathbold{\zeta}',\ldots,\mathbold{\zeta}^{(m)};
1,\frac{\mathbold{\Gamma}'}{\mathbold{\Gamma}},\ldots,\frac{\mathbold{\Gamma}^{(n)}}{\mathbold{\Gamma}}\Big)(z)=\\
&\frac{\mathcal{F}_{p_0}(\mathbold{\zeta},\mathbold{\zeta}',\ldots,\mathbold{\zeta}^{(m)};
\mathbold{\Gamma},\mathbold{\Gamma}',\ldots,\mathbold{\Gamma}^{(n)})(z)}{\mathbold{\Gamma}^{p_0}(z)}\not\equiv0.
\end{aligned}
\end{equation}

Set the digamma function $\mathbold{\psi}:=\frac{\mathbold{\Gamma}'}{\mathbold{\Gamma}}$, and deduce inductively $\frac{\mathbold{\Gamma}^{(j)}}{\mathbold{\Gamma}}=\mathbold{\psi}^j\big[1+(c_j+\varepsilon_j)\frac{\mathbold{\psi}'}{\mathbold{\psi}^2}\big]$ with $c_j=c_{j-1}+j-1=\frac{j(j-1)}{2}$ and $\varepsilon_j=\varepsilon_{j-1}+\frac{\varepsilon_{j-1}'}{\mathbold{\psi}}
+(c_{j-1}+\varepsilon_{j-1})\big[(j-3)\frac{\mathbold{\psi}'}{\mathbold{\psi}^2}+\frac{\mathbold{\psi}''}{\mathbold{\psi}\mathbold{\psi}'}\big]$ for $c_0=c_1=0$ and $\varepsilon_0=\varepsilon_1=0$.
See Han, Liu and Wang \cite{HLW} for the detailed calculations.

By virtue of \eqref{Eq02}, \eqref{Eq04} and \eqref{Eq07}, we get, for $h_j:=c_j+\varepsilon_j$ and $z\in\mathbf{C}$,
\begin{equation}\label{Eq08}
\begin{aligned}
&\mathcal{F}_{p_0}\Big(\mathbold{\zeta},\mathbold{\zeta}',\ldots,\mathbold{\zeta}^{(m)};
1,\frac{\mathbold{\Gamma}'}{\mathbold{\Gamma}},\ldots,\frac{\mathbold{\Gamma}^{(n)}}{\mathbold{\Gamma}}\Big)(z)\\
=&\sum_{q=0}^{M_{p_0}}\sum_{\lambda\in\Lambda_{p_0}\cap\Lambda^\star_q}F_\lambda(\vec{\mathbold{\zeta}}(z))
\Big(\frac{\mathbold{\Gamma}'}{\mathbold{\Gamma}}\Big)^{\lambda_1}(z)\prod^{n-1}_{j=2}\Big(\frac{\mathbold{\Gamma}^{(j)}}{\mathbold{\Gamma}}\Big)^{\lambda_j}(z)
\Big(\frac{\mathbold{\Gamma}^{(n)}}{\mathbold{\Gamma}}\Big)^{\lambda_n}(z)\\
=&\sum_{q=0}^{M_{p_0}}\mathbold{\psi}^{q+\chi(q)}(z)\prod^{n-1}_{j=2}\Big(1+h_j\frac{\mathbold{\psi}'}{\mathbold{\psi}^2}\Big)^{\lambda_j}(z)
\sum_{\lambda\in\Lambda_{p_0}\cap\Lambda^\star_q}F_\lambda(\vec{\mathbold{\zeta}}(z))\Big(1+h_n\frac{\mathbold{\psi}'}{\mathbold{\psi}^2}\Big)^{\lambda_n}(z),
\end{aligned}
\end{equation}
with $\vec{\mathbold{\zeta}}(z)$ henceforth being the abbreviation of the vector function $(\mathbold{\zeta},\mathbold{\zeta}',\ldots,\mathbold{\zeta}^{(m)})(z)$.

Note $\lambda_2,\ldots,\lambda_{n-1}$ depend only on $p_0,q$.
For fixed $p_0,q$, $\lambda$ is uniquely determined by $\lambda_n$, and vice versa.
Denote the largest $\lambda_n$ in $\Lambda_{p_0}$ by $N_{p_0}$.
Then, \eqref{Eq08} can be rewritten as
\begin{equation}\label{Eq09}
\begin{aligned}
&\mathcal{F}_{p_0}\Big(\mathbold{\zeta},\mathbold{\zeta}',\ldots,\mathbold{\zeta}^{(m)};
1,\frac{\mathbold{\Gamma}'}{\mathbold{\Gamma}},\ldots,\frac{\mathbold{\Gamma}^{(n)}}{\mathbold{\Gamma}}\Big)(z)\\
=&\sum_{q=0}^{M_{p_0}}\mathbold{\Pi}_q(z)\mathbold{\psi}^{q+\chi(q)}(z)
\sum^{N_{p_0}}_{r=0}F_{q,r}(\vec{\mathbold{\zeta}}(z))\Big(1+h_n\frac{\mathbold{\psi}'}{\mathbold{\psi}^2}\Big)^r(z).
\end{aligned}
\end{equation}
Here, $F_{q,r}(\vec{u}):=F_\lambda(\vec{u})$ for some $\lambda\in\Lambda_{p_0}$ if it exists while $F_{q,r}(\vec{u}):=0$ otherwise, and $\mathbold{\Pi}_q:=\prod^{n-1}\limits_{j=2}\big(1+h_j\frac{\mathbold{\psi}'}{\mathbold{\psi}^2}\big)^{\lambda_j}$.
We set $H:=h_n\frac{\mathbold{\psi}'}{\mathbold{\psi}^2}$ and expand $(1+H)^r$ to observe
\begin{equation}\label{Eq10}
\sum^{N_{p_0}}_{r=0}F_{q,r}(\vec{\mathbold{\zeta}}(z))(1+H)^r(z)=\sum_{r=0}^{N_{p_0}}G_{q,r}(\vec{\mathbold{\zeta}}(z))H^r(z).
\end{equation}
Here, for fixed $p_0,q$, the functions $F_{q,r}(\vec{u}),G_{q,r}(\vec{u})$ satisfy the relations
\begin{equation}\label{Eq11}
\begin{aligned}
&G_{q,N_{p_0}}(\vec{u})=F_{q,N_{p_0}}(\vec{u}),\\
&G_{q,N_{p_0}-1}(\vec{u})=F_{q,N_{p_0}-1}(\vec{u})+C^{N_{p_0}-1}_{N_{p_0}}F_{q,N_{p_0}}(\vec{u}),\\
&\ldots\\
&G_{q,0}(\vec{u})=F_{q,0}(\vec{u})+F_{q,1}(\vec{u})+\cdots+F_{q,N_{p_0}}(\vec{u}),
\end{aligned}
\end{equation}
so that $\left\{F_{q,r}(\vec{u})\right\}$ and $\left\{G_{q,r}(\vec{u})\right\}$ are mutually uniquely representable of each other.

Summarizing all the preceding discussions \eqref{Eq08}, \eqref{Eq09} and \eqref{Eq10} leads to
\begin{equation}\label{Eq12}
\begin{aligned}
&\mathcal{F}_{p_0}\Big(\mathbold{\zeta},\mathbold{\zeta}',\ldots,\mathbold{\zeta}^{(m)};
1,\frac{\mathbold{\Gamma}'}{\mathbold{\Gamma}},\ldots,\frac{\mathbold{\Gamma}^{(n)}}{\mathbold{\Gamma}}\Big)(z)\\
=&\sum_{q=0}^{M_{p_0}}\mathbold{\Pi}_q(z)\mathbold{\psi}^{q+\chi(q)}(z)\sum_{r=0}^{N_{p_0}}G_{q,r}(\vec{\mathbold{\zeta}}(z))H^r(z).
\end{aligned}
\end{equation}

Recall \eqref{Eq07}, and notice the $G_{q,r}(\vec{u})$ are continuous functions over $\mathbf{C}^{m+1}$.
Let $G_{q_0,r_0}(\vec{u})$ be the first nonzero term in the ordered sequence of functions as follows
\begin{equation}
\begin{aligned}
&G_{M_{p_0},0}(\vec{u}),G_{M_{p_0}-1,0}(\vec{u}),\ldots,G_{0,0}(\vec{u}),G_{M_{p_0},1}(\vec{u}),\ldots,G_{0,1}(\vec{u}),\ldots,\\
&G_{M_{p_0},N_{p_0}-1}(\vec{u}),\ldots,G_{0,N_{p_0}-1}(\vec{u}),G_{M_{p_0},N_{p_0}}(\vec{u}),\ldots,G_{1,N_{p_0}}(\vec{u}),G_{0,N_{p_0}}(\vec{u}).\nonumber
\end{aligned}
\end{equation}
In view of the finiteness of indices, one certainly can find a constant $C_0>1$ and a sufficiently small subset $\mathbf{\Omega}$ of $\mathbf{C}^{m+1}$ such that, after appropriate rescaling if necessary, $\n{G_{q_0,r_0}(\vec{u})}\geq1$ and $\n{G_{q,r}(\vec{u})}\leq C_0$ uniformly for all $\vec{u}\in\mathbf{\Omega}\varsubsetneq\mathbf{C}^{m+1}$ and for every $0\leq q\leq M_{p_0}$ and $0\leq r\leq N_{p_0}$.
Then, by Proposition \ref{Prop1}, there exists a sequence of real numbers $\left\{y_k\right\}_{k=1}^{+\infty}$ with $\n{y_k}\to+\infty$ such that $\gamma(y_k)\in\mathbf{\Omega}\varsubsetneq\mathbf{C}^{m+1}$ when $x=\frac{3}{4}$.
So, for $z_k:=\frac{3}{4}+\mathtt{i}y_k\in\mathbf{C}$, one has
\begin{equation}\label{Eq13}
\n{G_{q_0,r_0}(\vec{\mathbold{\zeta}}(z_k))}\geq1~\text{and}~\n{G_{q,r}(\vec{\mathbold{\zeta}}(z_k))}\leq C_0
\end{equation}
uniformly for all indices $q=0,1,\ldots,M_{p_0}$ and $r=0,1,\ldots,N_{p_0}$.

Next, the classical result of Stirling \cite[p.151]{Ti} states that
\begin{equation}
\log\mathbold{\Gamma}(z)=\Big(z-\frac{1}{2}\Big)\log z-z+\frac{1}{2}\log(2\pi)+\int_0^{+\infty}\frac{[t]-t+\frac{1}{2}}{t+z}dt,\nonumber
\end{equation}
which combined with a version of the Lebesgue's convergence theorem leads to
\begin{equation}
\mathbold{\psi}(z)=\frac{\mathbold{\Gamma}'}{\mathbold{\Gamma}}(z)=\log z-\frac{1}{2z}-\int_0^{+\infty}\frac{[t]-t+\frac{1}{2}}{(t+z)^2}dt.\nonumber
\end{equation}
By virtue of a routine differentiation procedure, for all $j=1,2,\ldots$, one derives
\begin{equation}
\mathbold{\psi}(z)=\log z(1+o(1))~\text{and}~\mathbold{\psi}^{(j)}(z)=\frac{(-1)^{j-1}(j-1)!}{z^j}(1+o(1)),\nonumber
\end{equation}
and thus
\begin{equation}\label{Eq14}
\frac{\mathbold{\psi}'}{\mathbold{\psi}^2}(z)=\frac{1}{z(\log z)^2}(1+o(1))~\text{and}~
\frac{\mathbold{\psi}''}{\mathbold{\psi}\mathbold{\psi}'}(z)=-\frac{1}{z\log z}(1+o(1)),
\end{equation}
uniformly on $\mathbf{D}:=\left\{z:-\frac{\pi}{6}\leq\arg z\leq\frac{\pi}{6}\right\}$, provided $\n{z}$ is sufficiently large.
Notice \eqref{Eq14} implies $\varepsilon_j=-\frac{j(j-1)(j-2)}{6}\frac{1}{z\log z}(1+o(1))$, so that, for an $\bm{o}(1)$ depending only on $p_0,q$,
\begin{equation}\label{Eq15}
\mathbold{\Pi}_q(z)=\prod^{n-1}\limits_{j=2}\Big(1+h_j\frac{\mathbold{\psi}'}{\mathbold{\psi}^2}\Big)^{\lambda_j}(z)=(1+\bm{o}(1))
\end{equation}
uniformly on $\mathbf{D}$ for sufficiently large $\n{z}$.
See \cite{HLW} for the detailed computations.

From now on, we shall focus entirely on $\left\{z_k\right\}^{+\infty}_{k=1}\varsubsetneq\mathbf{D}$ and observe that
\begin{equation}
\mathbold{\Pi}_q(z_k)\mathbold{\psi}^{q+\chi(q)}(z_k)G_{q,r}(\vec{\mathbold{\zeta}}(z_k))H^r(z_k)\nonumber
\end{equation}
is equal to, in view of $H(z_k)=h_n(z_k)\frac{\mathbold{\psi}'}{\mathbold{\psi}^2}(z_k)=\frac{n(n-1)}{2z_k(\log z_k)^2}(1+o(1))$ and \eqref{Eq15},
\begin{equation}
\Big[\frac{n(n-1)}{2}\Big]^rG_{q,r}(\vec{\mathbold{\zeta}}(z_k))\frac{(\log z_k)^{q+\chi(q)-2r}}{(z_k)^r}(1+\bm{o}(1))\nonumber
\end{equation}
when $k\to+\infty$.
Here, the indices either satisfy $r=r_0$ with $0\leq q\leq q_0$ or satisfy $r_0<r\leq N_{p_0}$ with $0\leq q\leq M_{p_0}$.
As a result, the term
\begin{equation}
\mathbold{\Pi}_{q_0}(z_k)\mathbold{\psi}^{q_0+\chi(q_0)}(z_k)G_{q_0,r_0}(\vec{\mathbold{\zeta}}(z_k))H^{r_0}(z_k),\nonumber
\end{equation}
among all the possible terms appearing in \eqref{Eq12}, dominates in growth when $k\to+\infty$.
Actually, seeing \eqref{Eq02}, for sufficiently large $k$, if $r=r_0$ with $0\leq q<q_0$, one derives
\begin{equation}
\frac{\n{\log z_k}^{q+\chi(q)}}{\n{z_k}^{r_0}\n{\log z_k}^{2r_0}}\ll\frac{\n{\log z_k}^{q_0+\chi(q_0)}}{\n{z_k}^{r_0}\n{\log z_k}^{2r_0}},\nonumber
\end{equation}
while if $r_0<r\leq N_{p_0}$ with $0\leq q\leq M_{p_0}$, one instead derives
\begin{equation}
\begin{aligned}
&\n{\log z_k}^{q+\chi(q)}\leq\n{\log z_k}^{M_{p_0}+\chi(M_{p_0})+q_0+\chi(q_0)}\\
&\n{z_k}^r\n{\log z_k}^{2r}\gg\n{z_k}^{r_0}\n{\log z_k}^{M_{p_0}+\chi(M_{p_0})+2r_0}.\nonumber
\end{aligned}
\end{equation}
Therefore, for sufficiently large $k$, one concludes through \eqref{Eq07}, \eqref{Eq12} and \eqref{Eq13} that
\begin{equation}\label{Eq16}
\n{\mathcal{F}_{p_0}\Big(\mathbold{\zeta},\mathbold{\zeta}',\ldots,\mathbold{\zeta}^{(m)};
1,\frac{\mathbold{\Gamma}'}{\mathbold{\Gamma}},\ldots,\frac{\mathbold{\Gamma}^{(n)}}{\mathbold{\Gamma}}\Big)(z_k)}
\geq\frac{1}{3}\frac{\n{\log z_k}^{q_0+\chi(q_0)-2r_0}}{\n{z_k}^{r_0}}.
\end{equation}

When $p_0=L_\Lambda$, by the definition of $p_0$ and \eqref{Eq16}, it yields that
\begin{equation}
\begin{aligned}
&\mathcal{F}(\mathbold{\zeta},\mathbold{\zeta}',\ldots,\mathbold{\zeta}^{(m)};\mathbold{\Gamma},\mathbold{\Gamma}',\ldots,\mathbold{\Gamma}^{(n)})(z_k)\\
=\,&\mathbold{\Gamma}^{L_\Lambda}(z_k)\mathcal{F}_{L_\Lambda}\Big(\mathbold{\zeta},\mathbold{\zeta}',\ldots,\mathbold{\zeta}^{(m)};
1,\frac{\mathbold{\Gamma}'}{\mathbold{\Gamma}},\ldots,\frac{\mathbold{\Gamma}^{(n)}}{\mathbold{\Gamma}}\Big)(z_k)\neq0\nonumber
\end{aligned}
\end{equation}
for sufficiently large $k$, which contradicts the assumption \eqref{Eq03}.

When $p_0<L_\Lambda$, using the classical result of Titchmarsh \cite[p.151]{Ti}
\begin{equation}
\n{\mathbold{\Gamma}\Big(\frac{3}{4}+\mathtt{i}y\Big)}=\sqrt{2\pi}e^{-\frac{1}{2}\pi\n{y}}\n{y}^{\frac{1}{4}}(1+o(1))\nonumber
\end{equation}
as $\n{y}\to+\infty$ and seeing that $\frac{\n{\log z}^\imath}{\n{z}^\jmath\n{\log z}^{2\jmath}}\to0$ as $\n{z}\to+\infty$ if $\jmath>0$ for nonnegative integers $\imath,\jmath$, one easily observes, in view of \eqref{Eq04}, \eqref{Eq05}, \eqref{Eq07}, \eqref{Eq13} and \eqref{Eq16}, as $k\to+\infty$,
\begin{equation}
\begin{aligned}
&\n{\frac{\mathcal{F}(\mathbold{\zeta},\mathbold{\zeta}',\ldots,\mathbold{\zeta}^{(m)};\mathbold{\Gamma},\mathbold{\Gamma}',\ldots,\mathbold{\Gamma}^{(n)})(z_k)} {\mathbold{\Gamma}^{L_\Lambda}(z_k)}}\\
=&\n{\sum_{p=p_0}^{L_\Lambda}\frac{1}{\mathbold{\Gamma}^{L_\Lambda-p}(z_k)}\mathcal{F}_p\Big(\mathbold{\zeta},\mathbold{\zeta}',\ldots,\mathbold{\zeta}^{(m)};
1,\frac{\mathbold{\Gamma}'}{\mathbold{\Gamma}},\ldots,\frac{\mathbold{\Gamma}^{(n)}}{\mathbold{\Gamma}}\Big)(z_k)}\\
\geq&\,\frac{1}{6}\exp\Big(\frac{L_\Lambda-p_0}{2}\pi\n{y_k}\Big)(\n{y_k}^{\frac{1}{4}}\sqrt{2\pi})^{p_0-L_\Lambda}\frac{\n{\log z_k}^{q_0+\chi(q_0)-2r_0}}{\n{z_k}^{r_0}}\\
&-C_1\exp\Big(\frac{L_\Lambda-p_0-1}{2}\pi\n{y_k}\Big)\n{y_k}^{\frac{p_0+1-L_\Lambda}{4}}\n{\log z_k}^K\to+\infty\nonumber
\end{aligned}
\end{equation}
for a constant $C_1>0$ and an integer $K\geq0$ depending on $n,\Lambda$.
So, again
\begin{equation}
\mathcal{F}(\mathbold{\zeta},\mathbold{\zeta}',\ldots,\mathbold{\zeta}^{(m)};\mathbold{\Gamma},\mathbold{\Gamma}',\ldots,\mathbold{\Gamma}^{(n)})(z_k)\neq0\nonumber
\end{equation}
for sufficiently large $k$, which contradicts the assumption \eqref{Eq03}.

Summarizing all the preceding analyses leads to \eqref{Eq06}, which further implies that the function $\mathcal{F}_p(u_0,u_1,\ldots,u_m;v_0,v_1,\ldots,v_n)$ itself must vanish identically.

In fact, when $p=0$, by definition, $\mathcal{F}_0(u_0,u_1,\ldots,u_m;v_0,v_1,\ldots,v_n)=F(u_0,u_1,\ldots,u_m)$ is a continuous function of $m+1$ variables with constant coefficients.
Therefore, via Voronin \cite{Vo2}, $F(\mathbold{\zeta},\mathbold{\zeta}',\ldots,\mathbold{\zeta}^{(m)})(z)\equiv0$ yields $F(u_0,u_1,\ldots,u_m)\equiv0$ immediately.

From now on, assume $p>0$.
For simplicity, write $p=p_0$, and use the expression \eqref{Eq12} and all its associated notations.
We next prove $G_{q,r}(\vec{u})\equiv0$ for all the indices $q,r$.
To this end, we first show that each one of
\begin{equation}
G_{M_{p_0},0}(\vec{u}),G_{M_{p_0}-1,0}(\vec{u}),\ldots,G_{0,0}(\vec{u})\nonumber
\end{equation}
must be identically equal to zero.
Let's start with $G_{M_{p_0},0}(\vec{u})$, and suppose that $G_{M_{p_0},0}(\vec{u})\not\equiv0$.
Then, \eqref{Eq13}, or its resemblance for this newly chosen $p_0$, holds.
Among all the terms appearing in \eqref{Eq12}, the term
\begin{equation}
\mathbold{\Pi}_{M_{p_0}}(z_k)\mathbold{\psi}^{M_{p_0}+\chi(M_{p_0})}(z_k)G_{M_{p_0},0}(\vec{\mathbold{\zeta}}(z_k))\sim(\log z_k)^{M_{p_0}+\chi(M_{p_0})}\nonumber
\end{equation}
dominates in growth for large $k$, since $\frac{\n{\log z_k}^\imath}{\n{z_k}^\jmath\n{\log z_k}^{2\jmath}}\to0$ when $k\to+\infty$ if $\jmath>0$ for nonnegative integers $\imath,\jmath$.
Thus, analogous to \eqref{Eq16}, we deduce from \eqref{Eq12} and \eqref{Eq13} that
\begin{equation}
\begin{aligned}
&\mathcal{F}_{p_0}(\mathbold{\zeta},\mathbold{\zeta}',\ldots,\mathbold{\zeta}^{(m)};\mathbold{\Gamma},\mathbold{\Gamma}',\ldots,\mathbold{\Gamma}^{(n)})(z_k)\\
=\,&\mathbold{\Gamma}^{p_0}(z_k)\mathcal{F}_{p_0}\Big(\mathbold{\zeta},\mathbold{\zeta}',\ldots,\mathbold{\zeta}^{(m)};
1,\frac{\mathbold{\Gamma}'}{\mathbold{\Gamma}},\ldots,\frac{\mathbold{\Gamma}^{(n)}}{\mathbold{\Gamma}}\Big)(z_k)\neq0\nonumber
\end{aligned}
\end{equation}
for sufficiently large $k$, which, however, contradicts the conclusion \eqref{Eq06}.
As a consequence, one sees $G_{M_{p_0},0}(\vec{u})\equiv0$.
The next term in line is $G_{M_{p_0}-1,0}(\vec{u})$ with
\begin{equation}
\mathbold{\Pi}_{M_{p_0}-1}(z_k)\mathbold{\psi}^{M_{p_0}-1+\chi(M_{p_0}-1)}(z_k)G_{M_{p_0}-1,0}(\vec{\mathbold{\zeta}}(z_k))\sim(\log z_k)^{M_{p_0}-1+\chi(M_{p_0}-1)},\nonumber
\end{equation}
and we deduce that $G_{M_{p_0}-1,0}(\vec{u})\equiv0$ in exactly the same manner.
And so on and so forth, we observe that $G_{q,0}(\vec{u})\equiv0$ for each $q=0,1,\ldots,M_{p_0}$.
Next, after the elimination of $H$ in \eqref{Eq12}, we can repeat the preceding procedure for
\begin{equation}
G_{M_{p_0},1}(\vec{u}),G_{M_{p_0}-1,1}(\vec{u}),\ldots,G_{0,1}(\vec{u})\nonumber
\end{equation}
and observe that $G_{q,1}(\vec{u})\equiv0$ for each $q=0,1,\ldots,M_{p_0}$.
Continuing like this, $G_{q,r}(\vec{u})\equiv0$, and thus $F_{q,r}(\vec{u})\equiv0$ by \eqref{Eq10} and \eqref{Eq11}, for all indices $q,r$; so, $\mathcal{F}_{p_0}(u_0,u_1,\ldots,u_m;v_0,v_1,\ldots,v_n)\equiv0$ follows, which implies $\mathcal{F}(u_0,u_1,\ldots,u_m;v_0,v_1,\ldots,v_n)\equiv0$ in view of \eqref{Eq05}.
\end{proof}

\end{document}